\documentclass[12pt]{article} 
\usepackage{amssymb, amsmath, amsthm}
\usepackage{geometry} 
\usepackage{amsfonts} 
\usepackage{amsmath} 
\usepackage{color}
\usepackage{cancel}
\usepackage{mathrsfs}
\usepackage{yfonts}
\usepackage{accents}
\usepackage{commath}
\usepackage{relsize}
\usepackage{graphicx}
\usepackage[utf8]{inputenc}
\usepackage{cite}

    \newcommand{\R}{\mathbb R}
    
    \newcommand{\K}{{\mathcal K}}

    \newcommand{\sub}{\subseteq} 
    \newcommand{\U}{\mathcal{U}} 
   \newcommand{\eps}{\varepsilon}
   
\newtheorem{thm}{Theorem}[section]
\newtheorem{lem}[thm]{Lemma}
\newtheorem{rmk}[thm]{Remark}
\newtheorem{prop}[thm]{Proposition}
\newtheorem{defi}[thm]{Definition}
\newtheorem{example}[thm]{Example}

\newcommand{\sett}[2]{\left\{  #1  \;\middle\vert\; #2 \right\}}

\newcommand{\li}{L^\infty}

    \newcommand{\ct}[1]{\langle {#1}\rangle \lower.3ex\hbox{$_{t}$}}
    \newcommand{\lt}[1]{[ {#1}] \lower.3ex\hbox{$_{t}$}}
   \newcommand{\blim}{\mathcal{B}\Lim{\U}}

\newcommand{\N}{{\mathbb N}}

\newcommand{\s}{\mathcal{S}}
\newcommand{\F}{\mathcal{F}}
\newcommand{\0}{{\bf 0}}

\newcommand{\Lim}[1]{\raisebox{0.5ex}{\scalebox{0.9}{$\displaystyle \lim_{#1}\;$}}}
\newcommand{\Sup}[1]{\raisebox{0.5ex}{\scalebox{0.9}{$\displaystyle \sup_{#1}\;$}}}

\title{Non-trivial translation-invariant valuations on $L^\infty$}

\author{L. Cavallina}
\date{}
\begin{document}
\maketitle

\begin{abstract} 
Translation-invariant valuations on the space $\li(\R^n)$ are examined. We prove that such functionals vanish on functions with compact support. Moreover a rich family of non-trivial translation-invariant valuations on $\li(\R^n)$ is constructed through the use of ultrafilters on $\R$.  
\end{abstract}
\bigskip


\bigskip


\tableofcontents

\section{Introduction}
The concept of \emph{valuation} arises in a fairly natural way when trying to formalize the idea of ``how big something is''.
Take for examples two finite sets $A$ and $B$: a reasonable way to evaluate how big these sets are (whence the name ``valuation''), is counting their elements. We know that 
\begin{equation}\nonumber
|A\cup B|=|A|+|B|-|A\cap B|,
\end{equation}
as we must be careful not to count the elements in the  intersection twice.
\par
Functionals that satisfy that property are nowadays called \emph{valuations}, although the first mathematician to ever study them (Hugo Hadwiger) referred to them as ``Eik{\"o}rper-\\funktional'' (literally egg-body functional) in \cite{Had}. 
Hadwiger's work provided a complete and elegant characterization of a special class of ``regular'' valuations over the family of compact convex sets (shortly, \emph{convex bodies}) of the Euclidean $n$-dimensional space.  He restricted his attention to those valuations that are rigid motion invariant and continuous with respect to a certain metric (called \emph{Hausdorff metric}) and proved that they can be written uniquely as a linear combination of $(n+1)$ fundamental valuations (called \emph{quermassintegrals}). In the same proof, Hadwiger also proved that rigid motion invariant monotone increasing (with respect to set inclusion) valuations can be expressed as a linear combination (with non negative coefficients) of the quermassintegrals.\par
Since then many others have tried to expand or generalize the results of Hadwiger (by changing assumptions on invariance and regularity) and to classify valuations defined on other classes (usually function spaces).  


Recently the concept of valuations has been extended from collections of sets to those of functions as well. Let $\mathcal{C}$ be a class of functions defined on a set $X$ which take values in a lattice $Y$. 
If $u,v:X\to Y$ we set $(u\lor v)(x):=u(x)\lor v(x)$ and $(u\land v)(x):= u(x)\land v(x)$ for all $x\in X$.

A map $\mu:\mathcal{C}\to \R$ is said to be a  (real valued) valuation if 
$$
\mu(u\lor v)+\mu(u\land v)=\mu(u)+\mu(v)
$$    
holds true for all $u,v\in\mathcal{C}$ such that $u\lor v,u\land v$ also belong to $\mathcal{C}$.
Usually, an additional property is required, namely that $\mu(\0)=0$ for a certain function $\0\in\mathcal{C}$; notice that the role of $\0$ is usually palyed by the constant function with constant equal to $0$ although some alternatives are also possible (see \cite{noi} for a case where the function constantly equal to $\infty$ is used).

Valuations defined on the Lebesgue spaces $L^p(\R^n)$ and $L^p(S^n)$, $1\le p<\infty$, were studied by Tsang in \cite{Tsang-2010}. The results of Tsang have been extended to
Orlicz spaces by Kone, in \cite{Kone}. 
Valuations of different types (taking values in $\K^n$ or in spaces of matrices, instead of $\R$),
defined on Lebesgue, Sobolev and BV spaces, have been considered in \cite{Tsang-2011}, 
\cite{Ludwig-2011a}, \cite{Ludwig-2012}, \cite{Ludwig-2013}, \cite{Wang-thesis}, \cite{Wang} and \cite{Ober}
(see also \cite{Ludwig-2011} for a survey).

\medskip

Other interesting results are due to Wright, that in his PhD Thesis \cite{Wright} and subsequently in collaboration with Baryshnikov and Ghrist in \cite{BGW}, studied a rather different class,
formed by the so-called {\em definable functions}. We will not present the details of the construction of these functions, but we mention that the main result
of these works is a characterization of valuation as suitable integrals of intrinsic volumes of level sets. 

In this paper we will focus on the space $\li(\R)$, i.e. the set of functions $u:\R\to \R$ which are measurable and moreover satisfy 
$$
\norm{u}_\infty := \inf \sett{a\in\R}{m\del{u^{-1}\del{a,\infty}}=0}<\infty, 
$$
where $m$ denotes the Lebesgue measure on the real line. 
We recall that $\norm{\cdot}_\infty$ is a semi-norm and it becomes a norm whenever functions that coincide almost everywhere with respect to the Lebesgue measure are identified. 
From now on we will always refer to the identified space and require valuations on $\li(\R)$ to vanish on the function $\0$ which is constantly equal to 0. 

The aim of this paper is to find a family of non-trivial examples of valuations $\mu$ on $\li(\R)$ which are translation-invariant (i.e. $\mu(u)=\mu(u\circ T)$ for all $u\in\li(\R)$ and all translations $T$ of the real line).

A. Tsang in \cite{Tsang-2010} provided a complete classification of translation-invariant continuous valuations on $L^p(\R^n)$ for $1\le p <\infty$. Indeed he proved that every translation-invariant continuous valuation $\mu:L^p\to\R$ can be written as 
$$
\mu(u)=\int_{\R^n} (h\circ u)(x)\dif x,
$$
where $h:\R\to\R$ is a suitable continuous function with $h(0)=0$ which is subject to a growth condition depending on the number $p$.

The case $p=\infty$ is extremely different from those treated by Tsang and this paper will convince you of that. 
It is well know that functions in $L^p(\R^n)$ can be approximated by a converging sequence of $L^p$-functions with compact support (and hence if $\mu:L^p(\R^n)\to\R$ is a continuous valuation, knowing how to compute $\mu$ on $L^p$-functions with compact support is equivalent to knowing $\mu$ itself).
The same thing is clearly not true for $\li(\R)$ in general. Moreover, a first important result of the present paper (see Proposition \ref{vanish}) consists in showing that translation-invariant valuations on $\li(\R)$ simply vanish on functions with compact support. 
  
This raises the question of whether we could actually exhibit an example of non-trivial translation-invariant valuation on $\li(\R)$ (a valuation is said to be {\em non-trivial} if it does not assign the value $0$ to every function). A non-constructive answer to this question will be given in sections from 4 to 8, where the axiom of choice is used to create translation-invariant extensions of the limit operator (called Banach limits) relying on a heavy use of ultrafilters.
It would indeed be interesting to prove whether any constructive example of such valuations can be given.  

Finally, the last section of this paper shows how the process described above could be adapted to work with valuations on $\li(\R^n)$ with $n>1$. \\

I would like to thank professor Andrea Colesanti for introducing me to the study of valuations and encouraging me to write this paper on my own while giving me precious advices.
 
\section{Preliminaries and notations}
As customary in mathematical analysis, $\N:=\set{1,2,\dots}$ will denote the set of natural numbers; moreover, the symbol ``$\sub$'' will be used for set-inclusion, while ``$\subset$'' is reserved for strict inclusion.
As we will mostly be working in dimension 1 (the $n$-dimensional case will only be treated in section \ref{ndim}), we will always use the simplified notation $\li$ instead of $\li(\R)$; moreover, given two functions $u,v\in\li$, we will write $u\equiv v$ to mean that they agree almost everywhere in $\R$ with respect to the Lebesgue measure. If $A\sub \R$, then $A^\complement $ will denote the complement $\R\setminus A$. For all real $c$, we set ${\bf c}$ to be the constant function ${\bf c}(x):=c$ for all $x\in\R$.

Let $A\sub \R$ and $u,v$ be real functions of one real variable; we will always write $\{u\in A\}$ instead of 
$\sett{x\in \R}{u(x)\in A}$ and $\{u=v\}$, $\{u\le v\}$ instead of $\sett{x\in\R}{u(x)=v(x)}$ and $\sett{x\in\R}{u(x)\le v(x)}$ respectively. 
 
Let $t\in\R$ and $u$ be a real function, then $\mathcal{T}_t u (x):= u(x-t)$, i.e. $\mathcal{T}_t u $ is the function whose graph is obtained translating that of $u$ by $t$ to the right; therefore we will say that a valuation $\mu:\li\to\R$ is translation-invariant whenever
$$
\mu(u)=\mu(\mathcal{T}_t u) \quad \forall u\in\li\, , \forall t\in\R.
$$
Given a valuation $\mu:\li\to\R$ we say that it is {\em monotone} if for all $u,v\in\li$ with $u\le v$ almost everywhere we have $\mu(u)\le \mu(v)$; moreover $\mu$ is said to be {\em continuous} if it is continuous with respect to the $\li$-metric, i.e. for all $\del{u_n}_{n\in\N}\subset \li$ such that $\Lim{n\to\infty} \norm{u_n-u}_\infty =0$ we have $\Lim{n\to\infty}\mu(u_n)=\mu(u)$.
 
We recall that the concept of limit at $\pm\infty$ can be definied for functions in $\li$ as follows:
$$
\lim_{x\to\infty} u(x)= l \quad \text{if} \quad \forall \eps>0\; \exists a\in\R :\; u(x)\in B_\eps (l) \text{ for almost every } x>a. 
$$
Limits for $x\to -\infty$ are defined by replacing the ``$x>a$'' above with ``$x<a$''.
Notice that such limits are not defined for all functions in $\li$ (extending them in a suitable way will be the main task of this paper); however, if $\Lim{x\to\infty} u(x)$ exists, it is equal to $\Lim{x\to\infty} \mathcal{T}_t u(x)$ for all real $t$ (in other words, limits are translation-invariant functionals).

Moreover, in order to aid the reader in understanding at a glance the hierarchical depth of the entities we use, throughout this paper we will stick to the following habit:
\begin{itemize}
\item real numbers will be denoted by lower case letters (e.g. $a,b,c,\eps,\dots$),

\item sets of real numbers (i.e. subsets of $\R$) will be denoted by Latin capital letters (e.g. A,B,C,\dots),

\item families of sets of real numbers (i.e. subsets of $\mathcal{P}(\R)$) will be denoted by Latin calligraphic capital letters  (e.g. $\mathcal{A},\mathcal{B},\mathcal{C},\dots$),

\item families of families of sets of real numbers (that is the deepest we will need to go in this paper, i.e. subsets of $\mathcal{P}(\mathcal{P}(\R))$) will be denoted by Latin script-style capital letters (e.g. $\mathscr{A},\mathscr{B},\mathscr{C},\dots $). 

\end{itemize}

\section{Isolating the tail parts}\label{isoltails}
In this section we will make rigorous what stated in the introduction, i.e. that translation-invariant valuations on $\li$ (unlike the related case concerning $L^p(\R^n)$ with $1\le p< \infty$) only care about the behaviour at infinity of the functions they are valuating.

We will say that two functions $u,v\in\li$ have disjoint support if $u\cdot v\equiv \0$. 
The following lemma shows that addition distributes with respect to $\lor$ and $\land$ when functions with disjoint support come into play.
\begin{lem}\label{distr}
If $u,v\in\li$ have disjoint support, then 
\begin{equation}\nonumber
\begin{split}
(u+v)\lor \0 \equiv u\lor \0 + v\lor \0, \\
(u+v)\land \0 \equiv u\land \0 + v\land \0. 
\end{split}
\end{equation}
\end{lem}
\begin{proof}
Consider $x\in\R$ such that $u(x)=0$, hence
$$
\left(u(x)+v(x) \right)\lor 0 = v(x)\lor 0= u(x)\lor 0 + v(x)\lor 0.
$$
If we consider $x$ such that $v(x)=0$ we reach the same conclusion, then, as $\{u\ne 0\}\cap \\ \{v\ne 0\}$ is a null set, 
we have that 
$$
(u+v)\lor \0 \equiv u\lor \0 + v\lor \0.
$$
Distributivity with respect to $\land$ is analogous.
\end{proof}

It is also easy to prove that valuations on $\li$ behave like additive functionals when the functions considered have disjoint support.

\begin{lem}\label{ddd}
Let $u,v\in\li$ have disjoint support; let $\mu:\li\to\R$ be a valuation. Then
$$
\mu(u+v)=\mu(u)+\mu(v).
$$
\end{lem}
\begin{proof}
Let us first prove the lemma under the additional hypothesis that $u,v\ge 0$ almost everywhere.
Consider $x$ such that $u(x)\cdot v(x)=0$, without loss of generality we may assume that $u(x)\ge0$ and $v(x)=0$.
Hence 
\begin{equation}\label{+lor}
u(x)+v(x)=(u\lor  v)(x).
\end{equation}
As \eqref{+lor} holds true also when $u(x)=0$ and $v(x)\ge0$, we conclude that 
$u+v\equiv u\lor v$.
In a similar way one could prove that $u\land v\equiv \0$, hence by the valuation property of $\mu$,
$$
\mu(u+v)=\mu(u\lor v)=\mu(u)+\mu(v)-\mu(\0)=\mu(u)+\mu(v).
$$
The same conclusion can be drawn when $u,v\le 0$ almost everywhere (in this case $u\lor v\equiv 0$ and $u\land v \equiv u+v$).

We will now prove the lemma in the general case.
Let $u,v\in\li$ be functions with disjoint support, then 
$$
\mu(u+v)=\mu(u+v)+\mu(\0)=\mu\left( (u+ v)\lor \0\right)+\mu\left( (u+v)\land \0\right);
$$
by Lemma \ref{distr}, 
$$
\mu\left( (u+v)\lor \0 \right)+\mu\left( (u+v)\land \0 \right)= \mu\left( u\lor \0 + v \lor \0 \right)+ \mu\left( u\land \0 +v\land \0\right).
$$
Notice that, as $u,v$ have disjoint support, so do $u\lor \0,v\lor \0$ and $u\land\0,v\land \0$.
Thus, employing what we have proven so far,
$$
\mu\left( u\lor \0 + v \lor \0 \right)+ \mu\left( u\land \0 +v\land \0\right)= \mu(u\lor \0) + \mu(v\lor \0) + \mu(u\land \0) + \mu(v\land \0)= \mu(u)+\mu(v).
$$

\end{proof}

We recall that a function $u\in\li$ is said to have {\em compact support} if it vanishes almost everywhere outside a bounded interval.

\begin{prop} \label{vanish}
Translation-invariant valuations on $\li$ vanish on all functions with compact support.
\end{prop}
\begin{proof}
Let $\mu:\li\to\R$ be a translation-invariant valuation on $\li$ and let $u\in\li$ have compact support. 
Without loss of generality we might suppose that $u$ vanishes almost everywhere outside $(0,n)$ for some $n\in\N$.
We ``prolong $u$ periodically to the right'' in the following way:
$$
\widetilde{u}(x):= u(r_x) \quad \text{for all } x\in\R,
$$
where $r_x$ is the only real number in $[0,n)$ satisfying $x=mn+r_x$ for some natural number $m$.

Notice that $u$ and $\mathcal{T}_n\widetilde{u}$ have disjoint support (see Figure \ref{prol}), moreover
$$
u+ \mathcal{T}_n\widetilde{u}= \widetilde{u}.
$$
\begin{figure}[h]
    \centering
    \includegraphics[width=0.7\textwidth]{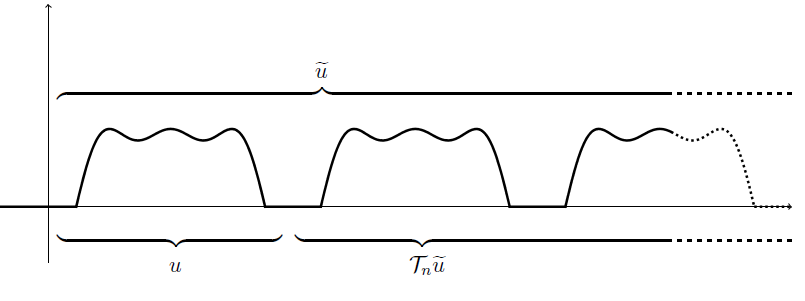}
   \caption{Prolonging $u$ periodically to the right}
\label{prol}
\end{figure}

As $\mu$ is a translation-invariant valuation, by Lemma \ref{ddd}
$$
\mu(\widetilde{u})=\mu(u+\mathcal{T}_n\widetilde{u})=\mu(u)+\mu(\mathcal{T}_n\widetilde{u})=\mu(u)+\mu(\widetilde{u}).
$$
We conclude that $\mu(u)=0$.

\end{proof}

Proposition \ref{vanish} tells us that translation-invariant valuations ``do not see'' what happens in a finite portion of $\R$ but rather only detect differences in asymptotic behaviours. 
We could also separate the two contributes of asymptotic behaviours at $+\infty$ and $-\infty$ respectively. 
Given a valuation $\mu$ on $\li$, we will introduce the follwing notation:
$$
\accentset{\rightarrow}{\mu}(u):= \mu(u\cdot \chi_{(0,\infty)})  \quad \text{and} \quad \accentset{\leftarrow}{\mu}(u):= \mu(u\cdot \chi_{(-\infty,0)}) \quad \forall u\in\li.
$$  
It is easy to prove that, if $\mu$ is a valuation, then also $\accentset{\rightarrow}{\mu}$ and $\accentset{\leftarrow}{\mu}$ are one (we will refer to them as the {\em right} and {\em left tail} of $\mu$, respectively). Moreover, if $\mu$ is translation-invariant, continuous or monotone, then its tails inherit these properties.

As, obviously, $\mu = \accentset{\rightarrow}{\mu}+\accentset{\leftarrow}{\mu}$ for every $\mu$, we could assume without loss of generality that $\mu$ coincides with its right tail.

\section{Banach limits of functions in $L^\infty$}
Banach limits were first introduced as a means to generalize the concept of limits of (bounded) real sequences.

The existence of a continuous, linear, shift-invariant functional $\phi: l^\infty \to \R$ (where $l^\infty$ denotes the set of bounded real sequences) that extends the usual limit is usually proven using the celebrated Hahn-Banach Theorem (see for instance \cite{hb}), and we refer to \cite{galvin} for a construction which employs the use of ultrafilters.
We emphasize that both constructions highly depend on the axiom of choice. 

In this section we will show that Banach limits could easily be defined on $L^\infty$ as well. Later in this paper Banach limits will be used to provide a family of non-trivial, monotone (or continuous), translation-invariant valuations on $L^\infty$.

We will now give the definition of a Banach limit for $L^\infty$. 
\begin{defi}[(right) Banach limit]\label{banlim}
A linear functional $\phi:L^\infty\to \R$ is called a (right) Banach limit if the following properties hold:
\begin{itemize}
\item $\phi(u)\ge 0$ for all $u\in L^\infty$ such that $u\ge 0$ almost everywhere, (positivity)
\item $\phi(\mathcal{T}_t u)=\phi(u)$ for all real $t$, (translation-invariance)
\item $\phi(u)= \Lim{x\to \infty} u(x)$ whenever the latter is defined. (extension of the limit operator)
\end{itemize}
\end{defi}
\begin{rmk}\emph{
Left Banach limits might also be defined replacing the third property with $\phi(u)= \Lim{x\to -\infty} u(x)$ (i.e. left Banach limits extend limits to $-\infty$ of arbitrary $L^\infty$ functions).} 
\end{rmk}

An easy consequence of Definition \ref{banlim} is that Banach limits are continuous functionals (indeed, they belong to the dual, $({L^\infty})^*$).
\begin{prop} Banach limits are continuous
\end{prop}
\begin{proof}
Let $u\in L^\infty$, if we set $m:=\norm{u}_\infty$ we have that $-\boldsymbol{m}\le u\le \boldsymbol{m}$ almost everywhere.
That implies that $\boldsymbol{m}-u\ge 0$ almost everywhere. Positivity of $\phi$ yields $\phi(\boldsymbol{m}-u)\ge 0$, which in turn implies (by linearity) that $\phi(\boldsymbol{m})\ge \phi(u)$. Then, as $\phi(\boldsymbol{m})=m$ (remember that $\phi$ extends the usual limit operator), we get $m\ge \phi(u)$. 
Proceeding in an analogous way we get $-m\le \phi(u)\le m$ which can be rewritten as $|\phi(u)|\le m=\norm{u}_\infty$. We conclude that $\phi$ is a continuous linear functional as claimed. 
\end{proof}
 
Sections \ref{lebult>ulim}, \ref{ulim>banlim} and \ref{existuf} are devoted to carefully showing  that there actually extists a linear functional satisfying the properties listed in Definition \ref{banlim}. 
The first step in this process consists in dividing the family of subsets of $\R$ into two disjoint families of ``small sets'' and ``large sets''.

\begin{defi}[Filter on $\R$]
A family $\mathcal{F}$ of subsets of $\R$ is called a \emph{filter} if the following holds:
\begin{itemize}
\item[{\bf F1)}] $\R\in \F$,
\item[{\bf F2)}] $\emptyset\notin \F$,
\item[{\bf F3)}] $\forall A,B\in\F$ we have $A\cap B\in \F$,
\item[{\bf F4)}] $\forall A\in \F$ and $\forall B\sub \R$ with $A\sub B$ then $B\in\F$. 
\end{itemize} 
\end{defi}  
A filter on $\R$ might be thought as a family of ``large sets''.
Consider a filter $\F$ and an arbitrary $A\sub \R$: it is immediate to realize that $A$ and $A^\complement$ cannot both belong to $\F$, because, if this were the case we would get $\emptyset =A\cap A^\complement \in \F$, which is a contradiction (by {\bf F2}); on the other hand it could happen that neither $A$ nor $A^\complement$ belong to $\F$ (consider for instance the trivial filter $\{\R\}$). 
\begin{defi}[Ultrafilter on $\R$]
A filter $\U$ on $\R$ is said to be an \emph{ultrafilter} if, for all $A\sub \R$, either $A\in\U$ or $A^\complement \in\U$.
\end{defi}
From now on the words ``(ultra)filter'' will be used as a synonym of ``(ultra)filter \underline{on $\R$}'' unless otherwise specified. 

In the present paper we are interested in a particular class of ultrafilters, namely ``right Lebesgue-ultrafilters''.

\begin{defi}[Lebesgue, right and left ultrafilters]
An ultrafilter $\U$ is a Lebesgue-ultrafilter if, for all Lebesgue-null set $N\sub \R$ we have $N^\complement \in\U$.
Moreover, an ultrafilter $\U$ is said to be a right (respectively, \emph{left}) ultrafilter if $(a,\infty)$ (respectively, $(-\infty,a)$) belongs to $\U$ for all $a\in\R$.
\end{defi} 
\begin{rmk}{\em
No ultrafilter can be simultaneously right and left, because of {\bf F3} and {\bf F2}.
}
\end{rmk}
Notice that the existence of right Lebesgue-ultrafilters (and actually of ultrafilters in general) cannot be taken for granted (as it will be shown in section \ref{existuf}, some more definitions and Zorn's lemma are essential to prove it).
Let us leave the problem concerning the existence of right Lebesgue-ultrafilters aside for a moment and let us assume that there actually exists at least one (which will be called $\U$), under this assumption it is fairly easy to create a linear functional that satisfies all the properties of Definition \ref{banlim} except for translation-invariance if we assume the existence of a right Lebesgue ultrafilter.

\section{Construction of ultralimits through Lebesgue-ultrafilters}\label{lebult>ulim}
\begin{defi}[(Right) ultralimits on $L^\infty$]\label{ulimdef}
Let $\U$ be a (right) Lebesgue-ultrafilter and let $u\in L^\infty$.
If there exists a real number $l$ such that
$$\forall \eps>0 \text{ there holds } \{x\in\R \mid |u(x)-l|<\eps \}=\{u \in B_\eps(l)\}\in \U,$$
we say that $l$ is the \emph{(right) ultralimit of $u$ with ultrafilter $\U$} (\emph{$\U$-limit of $u$}, for short), and we write
$$\lim_\U u=l.$$ 
\end{defi}

At this point we cannot conclude yet that, given a Lebesgue-ultrafilter, the corresponding ultralimit is well-defined (i.e. for every $u\in L^\infty$ there exists one and only one $l\in\R$ such that $\Lim{\U}u=l$, which is moreover independent of the representative chosen).

\begin{prop} Given a Lebesgue-ultrafilter $\U$, the functional $$\lim_\U:L^\infty\to\R$$ introduced in Definition \ref{ulimdef} is well defined.
\end{prop}
\begin{proof} 

Let us prove uniqueness first.
Suppose that there exists a function $u$ in $L^\infty$ such that 
$$ \lim_\U u= l_1 \text{ and } \lim_\U u = l_2 \text{\quad with } l_1\ne l_2.$$
As $l_1\ne l_2$ there exists $\eps>0$ small enough such that $B_\eps (l_1)\cap B_\eps (l_2)=\emptyset$.
Now, since both $l_1$ and $l_2$ are $\U$-limits of $u$ we have that 
$$
\{u\in B_\eps(l_1)\}\in \U,  \quad \{u\in B_\eps(l_2)\}\in \U.
$$
By {\bf F3}  we have

$$
\U\ni \{u\in B_\eps (l_1)\}\cap \{u\in B_\eps (l_2)\}= \{u\in B_\eps(l_1)\cap B_\eps (l_2)\}=\{u\in \emptyset\}=\emptyset;
$$
which is a contradiction.

Before proving existence, we will show that the set of admissible $\U$-limits of a fixed $u$ is a bounded set (although, as far as we know at this stage, it might be empty!).
Consider $u\in L^\infty$, by definition there exists $b>0$ such that $\{u\notin [-b,b]\}$ is a Lebesgue-null set. 
We claim that, if $\Lim{\U} u = l$ then $l\in [-b, b]$. 
Suppose that $\Lim{\U} u = l $ with $l\notin [-b,b]$ instead, therefore we can find $\eps>0$ small enough such that $B_\eps (l)\cap [-b,b]=\emptyset$.
As $\Lim{\U} u = l $ we have that $\{u\in B_\eps (l)\}\in \U$; moreover $\{u\in [-b,b]\}\in\U$ (because $\{u\in [-b,b]\}^\complement$ has Lebesgue measure zero and $\U$ is a Lebesgue-ultrafilter).
Putting all this together we get 
$$
\U\ni \{u\in B_\eps  (l)\}\cap \{u\in [-b,b]\}=\emptyset;
$$
contradiction.

We are ready to prove existence. Consider $u\in \li$ and $b>0$ as before. Suppose that the set of all admissible $\U$-limits of $u$ is empty. From the previous considerations this can be restated as
\begin{equation}\label{noneee}
\forall l \in [-b,b] \; \exists \eps_l>0 : \quad \{u\in B_{\eps_l}(l)\}\notin\U.
\end{equation}
As $[-b,b]$ is compact, we can extract a finite family of $l$'s, namely $l_1,\dots, l_n$ such that
$$
[-b,b]\sub \bigcup_{i=0}^n B_i,
$$
where we have set $B_i:=B_{\eps_{l_i}}(l_i)$.
As $\{u\in [-b,b]\}$ is the complement of a null set, we have that $\{u\in [-b,b]\}\in \U$ and thus, by {\bf F4}, $\{u\in\bigcup_{i=1}^n B_i\}\in \U$.
Hence $\bigcup_{i=1}^n\{u\in B_i\}\in\U$, which in turn implies that one of the sets $\{u\in B_i\}$ ($i=1,\dots,n$) belongs to $\U$ (this is due to the dual of {\bf F3}).
That is a contradiction, as we had supposed, by \eqref{noneee} that none of the $\{u\in Bi\}$'s  was a member of $\U$.

Finally, let $u\equiv v$ (in particular $\{u=v\}\in\U$).
Suppose $\Lim{\U} u = l$; we claim that also $\Lim{\U} v=l$.
Take an arbitrary $\eps>0$. We have that $\{u\in B_\eps (l)\}\in\U$. We need to show that also $\{v\in B_\eps (l)\}\in\U$. As a matter of fact 
$$
\{v\in B_\eps (l)\}\supseteq \{v\in B_\eps (l)\}\cap \{v=u\}= \overbrace{\underbrace{\{u\in B_\eps (l)\}}_{\in\U}\cap \underbrace{\{v=u\}}_{\in\U}}^{\in\U};
$$ 
thus $\{v\in B_\eps (l)\}\in\U$ (by {\bf F4}) and therefore $\Lim{\U} v= l$ as claimed.
\end{proof}

Having proven that ultralimits are well defined, we are ready to show some basic properties, namely linearity, positivity and extension of the usual limit.

\begin{prop} 
Let $\U$ be a right Lebesgue-ultrafilter. Then
\begin{itemize}
\item $\Lim{\U} (\alpha u+\beta v)=\alpha\, \Lim{\U} u + \beta\, \Lim{\U} v \quad \forall \alpha,\beta \in\R \; \forall u,v\in\li$ (linearity);
\item $\Lim{\U} u\ge 0$ if $u\ge 0$ almost everywhere (positivity);
\item $\Lim{\U} u = \Lim{x\to \infty} u(x)$ if the latter exists (extension of the usual limit).
\end{itemize}
\end{prop} 
\begin{proof}

Let us prove additivity first.
Let $u,v\in\li$ and set $\Lim{\U} u=:l_1$, $\Lim{\U} v=: l_2$. 
Take an arbitrary $\eps >0$, then 
$$
\{u+v\in B_\eps(l_1+l_2)\}\supseteq \underbrace{\{u\in B_{\eps/2}(l_1)\}}_{\in\U}\cap \underbrace{\{v\in B_{\eps /2}(l_2)\}}_{\in \U};
$$
therefore $\{u+v\in B_\eps (l_1+l_2)\}\in\U$ and additivity is proven.
Consider now $u\in\li$ and $\alpha \in\R$. Set $\Lim{\U} u = l$; we want to show that $\Lim{\U} (\alpha u)= \alpha\, \Lim{\U} u$.
The case $\alpha =0$ will be treated separatedly: 
$$\lim_\U 0\, u = \lim_\U \boldsymbol{0}=0=0\, \lim_\U u;$$
as a matter of fact, $\forall \eps>0$ we have $\{\boldsymbol{0}\in B_\eps (0)\}=\R\in\U$.
Let now $\alpha \ne 0$. For all $\eps>0$ 
$$\{\alpha u\in B_\eps (\alpha l)\}= \{u\in B_{\eps /\alpha }(l)\}\in\U.$$

Let us now prove positivity.
Take $u\in\li$ with the property that $u\ge 0$ almost everywhere and suppose $\Lim{\U} u =l<0$.
There exists $\eps>0$ small enough such that $B_\eps (l)\sub (-\infty,0)$.
Therefore 
$\{u\in B_\eps (l)\}\sub \{u\in (-\infty,0)\}$. As $u\ge 0$ almost everywhere, being $\U$ a Lebesgue ultrafilter, the set $\{u\in (-\infty, 0)\}$ does not belong to $\U$ (hence neither does $\{u\in B_\eps (l)\}$).

Finally we prove that $\Lim{\U}$ extends the usual limit. 
Let $u\in\li$ be a function such that $\Lim{x\to\infty}u(x)$ exists and is finite (we will call it $l$). We claim that also $\Lim{\U} u = l$.
Notice that for all $\eps>0$ there exists $a\in\R$  such that 
\begin{equation}\label{limmite}
|u(x)-l| <\eps \quad\text{for almost every }x> a.
\end{equation}
In other words $\{u\in B_\eps (l)\}\supseteq (a,\infty)\cap N^\complement$, where $N$ is the set where \eqref{limmite} fails to be true.
As $N$ is a null set, $N^\complement\in\U$, therefore 
$$
\underbrace{(a,\infty) \cap N^\complement}_{\in\U}\sub \{u\in B_\eps (l)\}\in\U. 
$$
\end{proof}

\begin{rmk}\em{
One might be tempted to define ``bilateral'' ultrafilters in order to extend limits as $|x|\to\infty$.
In fact, if we define an ultrafilter $\U$ to be {\em bilateral} whenever 
$$
\U\supset \sett{(a,b)^\complement}{ a,b\in\R, \; a < b},
$$
then ultralimits along bilateral Lebesgue-ultrafilters do actually extend the limit operator as $|x|\to\infty$.
The reason why bilateral ultrafilters are not interesting {\em per se} is that there actually does not exist a single bilateral ultrafilter which is neither a right or a left one. To prove it, consider a bilateral ultrafilter $\U$ and assume without loss of generality that $(0,\infty)\in\U$ (notice that either $(0,\infty)$ or $(-\infty,0)$ belongs to $\U$). This implies that every set of the form $(a,\infty)$ belongs to $\U$ (and hence $\U$ is a right ultrafilter). The assert is trivial if $a\le 0$; on the other hand if $a>0$ we have 
$$
(0,a]\cup (a,\infty)=(0,\infty)\in\U,
$$
which can happen only if $(a,\infty)\in\U$ (remember that $(0,a]\notin\U$ as $\U$ is bilateral).
Analogously, if $(-\infty,0)\in\U$, we deduce that $\U$ is a left ultrafilter. 
}
\end{rmk}

\begin{rmk}\label{notransinv}\em{
Notice that ultralimits are in general {\bf not} translation-invariant.

Let, for instance, 
$$
A:= \bigcup_{\mathbb{Z}\ni k \text{ odd}} [k,k+1),
$$
and consider its characteristic function 
$u:=\chi_A$.
It is clear that $u$ is a 2-periodic function, alternatively assuming the values 0 and 1. 
It can be proven that $\Lim{\U} u$ is either equal to 0 or 1 (although the exact value of the limit depends on the choice of the ultrafilter $\U$).
Indeed, $\forall l \notin \{0,1\}$ we can find $\eps >0$ small enough such that $B_\eps(l)\cap \{0,1\}=\emptyset$. We conclude that $\{u\in B_\eps (l)\}=\emptyset\notin \U$, hence $\Lim{\U} u \ne l$.
Even though we cannot determine a priori wheter $\Lim{\U} u$ equals 0 or 1, it can be shown that $\Lim{\U} \mathcal{T}_1 u \ne \Lim{\U} u$. 
Notice that $\mathcal{T}_1 u={\bf 1}-u$. This observation yields 
\begin{equation}\nonumber
\lim_\U u=\lim_\U \mathcal{T}_1 u = \lim_\U ({\bf 1}-u)= 1-\lim_\U u. 
\end{equation}
In other words, in order for the two limits to coincide, they both must be equal to $1/2$ (which contradicts $\Lim{\U} u \in \{0,1\}$).
}
\end{rmk}

\section{Construction of Banach limits through ultralimits}\label{ulim>banlim}
As we have noted in Remark \ref{notransinv}, ultralimits (as defined in Definition \ref{ulimdef}) are not Banach limits (in the sense of Definition \ref{banlim}) because they fail to be translation-invariant.

Here we will introduce a way to construct Banach limits using ultralimits. This technique is an adaptation to the continuum case of a standard construction of Banach limits for number sequences by means of ultrafilters on $\N$ (we refer to \cite{jerison} for more details).

First we will need the following preliminary construction.

\begin{defi}[Ces\`aro-like average of a real function]
Let $u:\R\to\R$ be a bounded measurable function. We define its Ces\`aro-like average $\widehat{u}:\R\to\R$ as 
\begin{equation}\nonumber
\widehat{u}(x) := 
\begin{cases}
u(0)  & \text{if }x=0, \\
\frac{1}{x} \int_0^x u(y) \dif y & \text{if } x\ne 0;
\end{cases}
\end{equation}
where we used the convention $\int_a^b f(y) \dif y:=-\int_{[b,a]}f(y) \dif y$ if $b<a$.
\end{defi}
\begin{rmk}\label{cesind}
\em{
Since the boundedness of $u$ implies that also $\widehat{u}$ is bounded (we can say more: in fact if $|u|\le m$ almost everywhere for some $m>0$, we can show that $|\widehat{u}|$ is bounded almost everywhere by the same constant), we can assert that Ces\`aro-like averages induce a continuous linear operator  $\widehat{\cdot}: \li\to\li$ in the obvious way.
}
\end{rmk}
The following lemma introduces some interesting properties of Ces\`aro-like averages.
\begin{lem} \label{cesprop} If $u\in\li$, then $u\ge 0$ almost everywhere implies $\widehat{u}\ge 0$ almost everywhere.
Moreover, if $\Lim{x\to\infty} u(x)=l$ then also $\Lim{x\to\infty} \widehat{u}(x)=l$.
\end{lem}
\begin{proof}
Let $u\ge 0$ almost everywhere. 
If $x>0$ then $\frac{1}{x}\int_0^x u(y) \dif y \ge 0$ as the integrand is non-negative almost everywhere, the limits of integration are ordered correctly (i.e. $0<x$) and $1/x$ is positive. 
On the other hand, if $x<0$, since the integrand is non-negative almost everywhere, the limits of integration are in the wrong order (i.e. $0>x$) and $1/x$ is negative, we can conclude that $\widehat{u}(x)\ge0$ for almost every $x$. 

Let now $\Lim{x\to\infty} u(x)=l$, we will prove that $\widehat{u}(x)$ converges to the same limit as $x\to\infty$.
As $u\in\li$, there exists $m>0$ sufficiently large for which $|u|< m$ almost everywhere.
As $\Lim{x\to \infty} u(x)=l$, for every $\eps >0$ there exists $a\in\R$ such that 
$$
l-\eps < u(x) < l+\eps \quad \text{for almost every } x>a. 
$$
Without loss of generality, we may restrict our attention to definitely large values of $x$  (in particular, assume $0<x=a+z$ for some positive $z$). 
Then we can write
\begin{equation}\nonumber
\begin{split}
\widehat{u}(x)=\frac{1}{x}\int_0^x u(y) \dif y = \frac{1}{a+z} \left( \int_0^a u(y)\dif y + \int_a^{a+z} u(y) \dif y \right) \\
\le \frac{1}{a+z}\left( \int_0^a m \dif y + \int_a^{a+z} (l+\eps) \dif y\right )= \frac{am}{a+z}+ \frac{z (l+\eps)}{a+z}\\   
= \frac{am}{x} + \frac{(x-a)(l+\eps)}{x}.
\end{split}
\end{equation}
In a completely analogous way the other inequality can be proven, so that, for $x$ sufficiently large we have
$$
-\frac{am}{x} + \frac{(x-a)(l-\eps)}{x}\le \widehat{u}(x)\le \frac{am}{x}+ \frac{(x-a)(l+\eps)}{x}.
$$
Thus,
$$
l-\eps \le \liminf_{x\to\infty} \widehat{u}(x)\le \limsup_{x\to \infty} \widehat{u}(x) \le l+\eps.
$$
By the arbitrariness of $\eps$ we conclude that $\Lim{x\to \infty} \widehat{u}(x)=l$.
\end{proof}

We are now ready to define the Banach limit associated to a right Lebesgue-ultrafilter.
\begin{defi}[Banach limits associated to an ultrafilter]\label{blimcostr}
Let $\U$ be a right Lebesgue-ultrafilter, then we can define the following functional $\mathcal{B}\Lim{\U}: \li\to\R$,
$$
\mathcal{B}\lim_\U u := \lim_\U \widehat{u} \quad \text{for all } u\in\li.
$$ 
\end{defi}
\begin{prop}\label{blimactually} The functional defined in Definition \ref{blimcostr} is actually a Banach limit in the sense of Definition \ref{banlim}.
\end{prop}
\begin{proof}
Fix a right Lebesgue-ultrafilter $\U$. 

First of all the functional $\blim$ is well defined, because, as stated in Remark \ref{cesind}, for every $u\in\li$, $\widehat{u}\in \li$ as well. 

Linearity is also obvious as $\blim$ is the composition of two linear operators.

Moreover, if $u\ge 0$ almost everywhere, then also $\widehat{u}\ge 0 $ almost everywhere by the first claim of Lemma \ref{cesprop} and therefore $\blim u = \Lim{\U} \widehat{u}\ge 0$.

Let now $\Lim{x\to\infty} u(x)=l$, the second claim of Lemma \ref{cesprop} implies that also $\Lim{x\to\infty}\widehat{u}(x)=l$ and finally, as ultralimits extend the usual limit operator, we conclude that $\blim u=\Lim{\U} z\cdot{u}=l$.

In order to prove translation-invariance, we will fix $u\in\li$ and $t\in\R$ and show that $\Lim{x\to\infty} \huge( \widehat{u}(x)-\widehat{\mathcal{T}_t u}(x)\huge)=0$. The claim $\blim u = \blim \mathcal{T}_t u$ follows by linearity and the extension property of $\blim$ that have just been proven.  
Consider $x>0$, then 
\begin{equation}\nonumber
\begin{split}
\widehat{u}(x)-\widehat{\mathcal{T}_t u}(x)= \frac{1}{x}\int_0^x u(y)\dif y - \frac{1}{x}\int_0^x u(y+t)\dif y 
=\frac{1}{x}\int_0^x u(y)\dif y -\frac{1}{x}\int_t^{x+t} u(y)\dif y \\
= \frac{1}{x}\left( \cancel{\int_0^x u(y)\dif y} -\int_t^0 u(y)\dif y - \cancel{\int_0^x u(y)\dif y} -\int_x^{x+t} u(y)\dif y   \right).
\end{split}
\end{equation}
Since $u\in\li$, we can find a positive constant $m$ such that $|u|\le m$ almost everywhere.
We deduce that 
$$
|\widehat{u}(x)-\widehat{\mathcal{T}_t u}(x)|\le \frac{2|t| m}{x};
$$
hence, letting $x\to\infty$, we get $\Lim{x\to\infty } \widehat{u}(x)-\widehat{\mathcal{T}_t u}(x)=0$.
\end{proof} 

We have finally proven the existence of Banach limits modulo the existence of right Lebesgue-ultrafilters.
In the following section we will show a construction of such ultrafilters that employs Zorn's lemma. 

\section{On the existence of right Lebesgue-ultrafilters}\label{existuf}

\begin{defi}[FIP]
A family $\mathcal{S}$ of subsets of a set $X$ is said to have the \em{finite intersection property} (\em{FIP} for short) if the intersection over any finite subcollection of $\mathcal{S}$ is not empty, in other words
$$
\forall n\in\N \; \forall D_1,\dots,D_n\in\mathcal{S} \quad \text{ we have } \bigcap_{i=1}^n D_n \ne \emptyset.
$$ 
\end{defi}

\begin{example}\label{NR}
The sets 
$$
\mathcal{N}:=\sett{B\sub \R}{B^\complement \text{ is a Lebesgue-null set }}, \quad
\mathcal{R}:=\sett{(a,\infty)}{a\in\R}
$$
are closed under finite intersections and do not contain the empty set as one of their elements; in particular this implies that they have the FIP.
\end{example}

\begin{lem}Let $\mathcal{S}$ be a non-empty family of subsets of $\R$ with the FIP, then the set 
$$
\langle \mathcal{S} \rangle := \{A\sub \R \,\mid\, \exists n\in\N : \, \exists D_1,\dots,D_n\in\mathcal{S}:\, \bigcap_{i=1}^n D_i \sub A \}
$$
is a filter (it is called the \emph{filter generated by $\mathcal{S}$}).
\end{lem}
\begin{proof}

It is clear that $\R\in\langle \mathcal{S}\rangle$. 
Moreover, suppose that $\emptyset\in\langle \mathcal{S}\rangle$, we would have some $D_1,\dots, D_n\in\mathcal{S}$ with
$$
\bigcap_{i=1}^n D_i\sub \emptyset \implies \bigcap_{i=1}^n D_i = \emptyset,
$$
which contradicts the FIP of $\s$.

Let $A,B\in\langle \s\rangle$, then $\bigcap_{i=1}^n D_i\sub A$ and $\bigcap_{j=1}^m F_j\sub B$ for some $D_1,\dots, D_n,F_1,\dots, F_m\in\s$.
This implies 
$$
\left( \bigcap_{i=1}^n D_i \right) \cap \left(\bigcap_{j=1}^m F_j \right)\sub A\cap B,
$$
hence $A\cap B\in\langle \s\rangle$.

Finally let $A\in\langle \s\rangle$ (which means $A\supseteq \bigcap _{i=1}^n D_i$ for some suitable $D_i$'s in $\s$) and consider $A\sub B\sub \R$. It is immediate to prove that also $B\in\langle \s\rangle$ as 
$$
\bigcap_{i=1}^n D_i\sub A\sub B.
$$ 
\end{proof}

In order to finally prove the existence of right Lebesgue-ultrafilters, we will need the following version of the axiom of choice, known as Zorn's lemma (we refer to \cite{naiveset} for more details).

\begin{lem}[Zorn's lemma]
If $\mathscr{A}$ is a partially ordered set such that every chain $\mathscr{B}\sub \mathscr{A}$ has an upper bound in $\mathscr{A}$, then $\mathscr{A}$ contains at least one maximal element. 
\end{lem} 

Before proceeding with our work, let us recall some basic definitions.

\begin{defi}[Partially ordered set]
Let $\mathscr{A}$ be a set; a relation $\le$ on $\mathscr{A}$ is said to be a partial order if it is {\em reflexive}, {\em antisymmetric} and {\em transitive}. Moreover $\le$ is called a {\em total order} if $\forall \mathcal{A},\mathcal{B}\in\mathscr{A}$ either $\mathcal{A}\le \mathcal{B}$ or $\mathcal{B}\le \mathcal{A}$ holds true.
\end{defi} 

\begin{defi}[Chain]
 Let $\mathscr{A}$ be a set partially ordered by $\le$, a subset $\mathscr{B}\sub \mathscr{A}$ is said to be a chain if $\le$ induces a total order relation on $\mathscr{B}$.
\end{defi}

\begin{defi}[Upper bound]
Let $\mathscr{A}$ be a set partially ordered by $\le$ and let $\mathscr{B}\sub \mathscr{A}$. 
We say that $\mathscr{B}$ has an {\em upper bound} in $\mathscr{A}$ if there exists $\mathcal{A}\in\mathscr{A}$  such that $\mathcal{B}\le\mathcal{A}$ for all $\mathcal{B}\in\mathscr{B}$.
\end{defi}

\begin{defi}[Maximal element]
Let $\mathscr{A}$ be a partially ordered set with order relation $\le$. An element $\mathcal{M}\in\mathscr{A}$ is a {\em maximal element} if for all $\mathcal{L}\in\mathscr{A}$ with $\mathcal{M}\le \mathcal{L}$ we have that necessarily $\mathcal{M}=\mathcal{L}$.
\end{defi}

\begin{rmk}
If $\mathscr{A}$ is a family of sets, then $\mathscr{A}$ is partially ordered by the set inclusion $\sub$.
\end{rmk}

\begin{thm}[Existence of right Lebesgue-ultrafilters]
There exists at least one right Lebesgue-ultrafilter on $\R$.
\end{thm}
\begin{proof}
The first step consists in proving the existence of at least one maximal filter containing $\mathcal{N}$ and $\mathcal{R}$ (see example \ref{NR}).
In the final step we will show that such a filter is indeed a right Lebegue-ultrafilter. 

Consider the following set
$$
\mathscr{A}:=\sett{\mathcal{F}}{\mathcal{F} \text{ is a filter on }\R \text{ and } \mathcal{N}\cup\mathcal{R}\sub \mathcal{F}}.
$$
We are going to apply Zorn's lemma to $\mathscr{A}$. 
We claim that $\mathscr{A}$ is non-empty (hence the empty chain has an upper bound in $\mathscr{A}$). In order to prove it, we will show that the set $\mathcal{N}\cup \mathcal{R}$ has the FIP. Take some $N_1,\dots, N_t\in\mathcal{N}$ and $R_1,\dots,R_s\in\mathcal{R}$; as both $\mathcal{N}$ and $\mathcal{R}$ are closed under finite intersections, as pointed out in Example \ref{NR}, without loss of generality we can restrict our attention to the case $t=s=1$.
Indeed, take $N\in\mathcal{N}$ and $(a,\infty)\in\mathcal{R}$ and suppose that $N\cap (a,\infty)=\emptyset$; in other words we are assuming $N\sub (-\infty,a]$, thus $N^\complement \supseteq (a,\infty)$, which is a contradiction, as $N^\complement$ is a null-set. As $\mathcal{N}\cup \mathcal{R}$ has the FIP, the filter $\langle\mathcal{N}\cup \mathcal{R} \rangle $ is well defined and thus $\mathscr{A}\ne\emptyset$.

Consider now an arbitrary non-empty chain $\mathscr{B}\sub \mathscr{A}$. We will prove that 
$$
\mathcal{G}:=\bigcup_{\mathcal{B}\in\mathscr{B}}\mathcal{B}\in \mathscr{A},
$$
so that in particular we will have shown that $\mathscr{B}$ has an upper bound in $\mathscr{A}$.
As $\mathcal{N}\cup \mathcal{R}\sub \mathcal{B}$ for all $\mathcal{B}\in\mathscr{B}$, we just need to prove that $\mathcal{G}$ is a filter. 
As $\R\in\mathcal{B}$ and $\emptyset\notin\mathcal{B}$ for all $\mathcal{B}\in\mathscr{B}$ we have that $\R\in\mathcal{G}$ and $\emptyset\notin \mathcal{G}$
Take now two arbitrary $A,B\in\mathcal{G}$, we have that $A\in\mathcal{B}_1$ and $B\in\mathcal{B}_2$ for some $\mathcal{B}_1,\mathcal{B}_2\in\mathscr{B}$; as $\mathscr{B}$ is a chain, without loss of generality we may assume that $\mathcal{B}_1\sub \mathcal{B}_2$, hence $A,B\in\mathcal{B}_2$. Since $\mathcal{B}_2$ is a filter, $A\cap B \in\mathcal{B}_2\sub\mathcal{G}$.
Let now $A\in\mathcal{G}$ and let $B\sub\R$ such that $A\sub B$; as $A\in\mathcal{G}$ we have that $A\in\mathcal{B}$ for some suitable $\mathcal{B}\in\mathscr{B}$, hence, as $\mathcal{B}$ is a filter, also $B\in \mathcal{B}$, hence $B\in\mathcal{G}$.
We have just proven that every chain $\mathscr{B}\sub \mathscr{A}$ (be it either empty or non-epty) has an upper bound in $\mathscr{A}$. Therefore Zorn's lemma ensures us the existence of a maximal element $\U\in\mathscr{A}$. We are now going to prove that $\U$ is an ultrafilter.

Take an arbitrary $A\sub\R$. If we show that either $A$ or $A^\complement$ belongs to $\U$ we can conclude that $\U$ is an ultrafilter.
Assume, towards a contradiction, that neither $A$ nor $A^\complement$ is an element of $\U$. Then we claim that $\U\cup \{A^\complement\}$ has the FIP. This can be proven by contradiction: suppose, without loss of generality, that here exists $U\in\U$
such that $U\cap A^\complement =\emptyset$. This means that $U\sub A$, thus $A\in\U$, leading to a contradiction. As $\U\cup \{A^\complement\}$ has the FIP, we can deduce that $\U\sub \langle \U\cup \{A^\complement\}\rangle$, hence, by maximality,
$\U=\langle \U\cup \{A^\complement\}\rangle\ni A^\complement$, which is against our previous assumptions.
We have therefore proven that $\U$ is an ultrafilter containing $\mathcal{N}$ and $\mathcal{R}$, i.e. it is a right Lebesgue-ultrafilter.
\end{proof} 

\section{Contruction of valuations on $\li$ through Banach limits}
With all this machinery at disposal it is fairly easy to provide some non-trivial examples of translation-invariant valuations on $\li$.
\begin{prop}\label{koko}
Let $\U$ be a right (or left) Lebesgue-ultrafilter, then the functional 
$$
\mathcal{B}\lim_\U : \li\to\R
$$
defined in Definition \ref{blimcostr} is a non-trivial, monotone, continuous, translation-invariant valuation.
\end{prop}
\begin{proof}
Let $u,v\in\li$. Now, as for all $x\in\R$, $\{u(x),v(x)\}=\{u(x)\lor v(x), u(x)\land v(x)\}$, 
we can write 
$$
u+v = u\lor v+u\land v.
$$
Therefore, by linearity:
$$
\mathcal{B}\lim_\U u + \mathcal{B}\lim_\U v = \mathcal{B}\lim_\U \left(u\lor v\right) + \mathcal{B}\lim_\U \left(u\land v\right).
$$
Moreover, as Banach limits extend the usual limit, $\blim \0=0$, in other words the functional $\blim$ is a valuation. Proposition \ref{blimactually} tells us that $\blim$ is monotone, continuous and translation-invariant; it is moreover non-trivial because $\blim {\bf 1}=1\ne 0$.
\end{proof}

The following lemma will help us extend the result of Proposition \ref{koko}.

\begin{lem}\label{muzukasii}
Let $f:\R\to\R$ be monotone non-decreasing (respectively, continuous), then the map
$$
\fullfunction{\bar{f}}{\li}{\li}{u}{f\circ u}
$$
is well defined and monotone non-decreasing (respectively, contiunuous in the $\li$-metric).
\end{lem}
\begin{proof}
Let $u\in\li$. It is simple to check that also $f\circ u$ belongs to $\li$ (i.e. the application is well defined); indeed $f$ is Borel measurable whenever it is monotone or continuous and hence so is $f\circ u$. 
To show that $f\circ u$ is bounded almost everywhere, recall that there exists some $m>0$ such that $-m\le u(x)\le m$ for almost every $x\in\R$.
If $f$ is monotone non decreasing, then 
$$
\inf_{[-m,m]}f \le f(u(x))\le \sup_{[-m,m]}f \quad \text{for almost every } x\in\R,
$$ 
where the extreme values above are actually attained in case $f$ is either monotone or continuous. Thus $f\circ u \in\li$ as claimed.

If $f$ is monotone, then $\bar{f}$ inherits the same property.
How continuity of $f$ translates into that of $\bar{f}$ is more delicate.
Suppose now that $f$ is continuous and take an arbitrary sequence $\del{u_n}_{n\in\N}$ of functions in $\li$ converging to $u\in\li$ in the $\li$-metric.\\
As $\norm{u_n}_\infty\to\norm{u}_\infty$, we deduce that the sequence $\del{u_n}_{n\in\N}$ is bounded, in particular there exists a real number $m$ such that
$$
-{\bf m}\le u_n \le {\bf m}, \quad -{\bf m}\le u \le {\bf m} \quad \text{for all $n\in\N$, almost everywhere in $\R$.}
$$
Moreover, since $f$ is continuous in $\R$, it is uniformly continuous in $\intcc{-m,m}$, i.e. 
\begin{equation}\label{unicont}
\forall \eps>0 \; \exists \delta>0 : \; \forall t,s\in [-m,m]; \quad \abs{t-s}<\delta \implies \abs{f(t)-f(s)}<\eps.
\end{equation}
We assumed that $\norm{u_n-u}_\infty\to 0$, in other words:
\begin{equation}\label{conv}
\forall \eta >0 \;\exists n_0 :\quad  \N\ni n>n_0 \implies \norm{u_n-u}_\infty < \eta. 
\end{equation}
Fix now $\eps>0$: we can find a suitable $\delta>0$ such that \eqref{unicont} is satisfied; moreover, if we set $\eta:=\delta$, using \eqref{conv} we find a natural number $n_0$ such that, $\forall n>n_0$ we have
$$
\abs{u_n(x)-u(x)}\le \norm{u_n-u}_\infty < \delta\quad \text{for almost every $x\in\R$.} 
$$     
 Hence (by \eqref{unicont}) for all $n>n_0$ the following holds:
$$
\abs{f(u_n(x))-f(u(x))}<\eps \quad \text{for almost every $x\in\R$}.
$$
Passing to the essential supremum yields $\norm{f(u_n)-f(u)}_\infty \le \eps$ definitely in $n$. In other words $f(u_n)\to f(u)$ in the $\li$-norm as claimed. 
\end{proof}

\begin{prop}\label{generalex}
Let $f:\R\to\R$ be either monotone non-decreasing or continuous and such that $f(0)=0$ and let $\U$ be a right (or left) Lebesgue-ultrafilter; then $\mu(\cdot):=\blim f(\cdot)$ is a translation-invariant valuation on $\li$.
Moreover, if $f$ is monotone (respesctively continuous), then also $\mu$ is monotone (respectively, continuous).
\end{prop}
\begin{proof}
Reasoning just like we did in Proposition \ref{koko} we have 
\begin{equation}\label{effe}
f(u)+f(v)=f(u\lor v)+f(u\land v)
\end{equation}
for all $u,v\in\li$ and all $f:\R\to\R$ as above (notice that all functions in \eqref{effe} belong to $\li$ by Lemma \ref{muzukasii}). Hence $\mu(u)+\mu(v)=\mu(u\lor v)+\mu(u\land v)$ for all $u,v\in\li$.

Moreover, as $f(0)=0$ by hypothesis, $\mu(\0)=\blim f(\0)= \blim \0 = 0$, whence $\mu$ is a valuation.
Translation-invariance of $\mu$ is implied by by that of Banach limits. 

By Lemma \ref{muzukasii}, monotonicity (respectively, continuity) of $f$ implies that also $\mu$ is monotone (respectively, continuous). 
\end{proof}

It is noteworthy that Proposition \ref{generalex} does not exhaust all the possibile translation-invariance valuations (be they either monotone or continuous). Nevertheless many other valuations can be formed, for instance the sum of a finite number of those introduced in Proposition \ref{generalex} (taking the Banach limits possibly with respect to distinct ultrafilters) or even the series of countably  many of them as we will see in the following proposition. 

\begin{prop}\label{serie}
Let $\set{\U_i}_{i\in\N}$ be a countable family of either right or left Lebesgue ultrafilters. 
Let $\set{f_i}_{i\in\N}$ be a family of real valued functions such that $f_i(0)=0$ for all $i\in\N$ and 
\begin{equation}\label{summab}
\sum_{i=1}^\infty \sup_{\intcc{-m,m}}\abs{f_i}<\infty \quad \text{for all }m>0. 
\end{equation}
Then the functional 
$$
\fullfunction{\mu}{\li}{\R}{u}{\sum_{i=1}^\infty \mathcal{B}\lim_{\U_i}f_i(u)}
$$
is a translation-invariant valuation.
Moreover, if the $f_i$'s are monotone (respectively, continuous) then so is $\mu$.
\end{prop}  
\begin{proof}
For the sake of notation, we will set $\phi_i(\cdot):=\mathcal{B}\Lim{\U_i}f_i(\cdot)$.
First of all, condition \eqref{summab} tells us that the functional $\mu$ is well defined, as the series $\sum_{i=1}^\infty \phi_i(u)$ is absolutely convergent for all $u\in\li$.
For obvious reasons $\mu$ is a translation-invariant valuation.

Let now the $f_i$'s be monotone and let $u,v\in\li$ with $u\le v$ almost everywhere. As the $\phi_i$'s are monotone (see Proposition \ref{generalex}) we have that 
$$
\sum_{i=1}^n \phi_i(u)\le \sum_{i=1}^n \phi_i(v)\quad \forall n\in\N,
$$
hence, passing to the limit as $n\to\infty$ we obtain $\mu(u)\le\mu(v)$.

Let the $f_i$'s be continuous this time and consider a sequence $\del{u_n}_{n\in\N}$ in $\li$ such that $u_n\to u\in\li$ in the $\li$-metric. As $\norm{u_n}_\infty\to\norm{u}_\infty$, we know that there exists a real number $m>0$ such that $m>\norm{u_n}_\infty$, $m> \norm{u}_\infty$ for all $n\in\N$; this implies that, for all $i,n\in\N$, $\abs{\phi_i(u_n)}\le \Sup{\intcc{-m,m}}\abs{f_i}$.
Moreover, as the $\phi_i$'s are continuous by Proposition \ref{generalex} and $\sum_{i=1}^\infty \Sup{\intcc{-m,m}}\abs{f_i}$ is convergent by hypothesis, the dominated convergence theorem for series implies that
$$
\lim_{n\to\infty}\sum_{i=1}^\infty \phi_i(u_n)=\sum_{i=1}^\infty \phi_i(u).
$$

\end{proof}

\section{The $n$-dimensional case}\label{ndim}
The choice of working in dimension 1 is mostly due to psychological reasons, rather than mathematical ones. 
Indeed right (or left) Banach limits of bounded real functions of one variable seemed to arise naturally as an easy generalization of Banach limits of bounded real sequences. Anyway it turned out that the almost totality of the proofs in these paper could be adapted to deal with valuations on $\li(\R^n)$ without much effort.

Translation-invariant valuations still vanish on functions of compact support even in the $n$-dimensional case. Indeed if, for all $m>0$ and hyperplane $H\subset \R^n$, we set 
$$
B_m(H):=\sett{x\in\R^n}{{\rm dist}(x,H)<m},
$$
we can state the following stronger result (compare with Proposition \ref{vanish}).
\begin{prop}
Let $\mu:\li(\R^n)\to \R$ be a translation-invariant valuation. If $u\in\li(\R^n)$ is such that $u$ vanishes almost everywhere outside $B_m(H)$ for some hyperplane $H\subset \R^n$ and some $m>0$, then $\mu(u)=0$.
\end{prop}
\begin{proof}[Sketch of the proof]
Prolong $u$ periodically in a direction which is orthogonal to $H$ and then conclude just like we did in Proposition \ref{vanish}. 
\end{proof}

In order to define Banach limits in the $n$-dimensional setting one could think of retracing the whole process shown in sections from 3 to 7 of this paper (i.e. proving the existence of suitable ultrafilters on $\R^n$, then proceeding to define ultralimits and finally Banach limits through an appropriate adaptation of the Ces\`aro-like average). It is in fact true that one can define ultrafilters and ultralimits just like we did in sections 4, 5 and 7, but even in this case, as ultralimits fail to be translation-invariant, one must resort to Ces\`aro-like averages anyways.
That being the case, the following definition of Ces\`aro-like average permits to obtain translation-invariant functionals, while not having to rely on ultrafilters on $\R^n$ for $n>1$.  
\begin{defi}[Ces\`aro-like average of a real function of several real variables]
Let $u:\R^n\to\R$ be a bounded measurable function.We define its Ces\`aro-like average $\widehat{u}:\R\to\R$ as 
\begin{equation}\nonumber
\widehat{u}(x) := \frac{1}{|B_x(0)|} \int_{B_x(0)} u(y) \dif y \quad \text{if } x>0,
\end{equation}
where $|B_x(0)|$ denotes the $n$-dimensional Lebesgue measure of the sphere;\\
while we set 
\begin{equation}\nonumber
\widehat{u}(x)=0 \quad \text{if } x\le 0.
\end{equation}
\end{defi}

It is now easy to prove that the just defined operator induces a continuous linear operator $\widehat{\cdot}: \li(\R^n)\to \li(\R)$. 
Moreover, following Lemma \ref{cesprop} as a guideline, it is straightforward to show that the operator above preserves order and limits, i.e. for all $u\in\li(\R^n)$ 
\begin{equation}\nonumber
\begin{gathered}
u\ge 0 \text{ almost everywhere} \implies \widehat{u}\ge 0 \text{ almost everywhere}; \\
\lim_{\abs{x}\to\infty} u(x)=l \implies \lim_{x\to\infty} \widehat{u}(x)=l .
\end{gathered}
\end{equation}

Therefore, for every fixed right Lebesgue-ultrafilter $\U$ we have a linear, order-preserving, continuous functional 
$$
\fullfunction{\blim}{\li(\R^n)}{\R}{u}{\blim \widehat{u}\,.}
$$

Translation-invariance can be proven as in Proposition \ref{blimactually}, namely, $\forall t\in\R^n$, $\forall u\in\li(\R^n)$, 
we need to show that $\Lim{x\to\infty}\widehat{\mathcal{T}_t u}(x)-\widehat{u}(x)= 0$.  
To that end, if we consider $x>0$, we get
\begin{equation}\nonumber
\begin{aligned}
\widehat{\mathcal{T}_t u}(x)-\widehat{u}(x)= \frac{1}{\abs{B_x(0)}}\int_{B_x(0)}\mathcal{T}_t u(y)\dif y - \frac{1}{\abs{B_x(0)}}\int_{B_x(0)}u(y)\dif y=   \\
\frac{1}{\abs{B_x(0)}} \left(  \int_{B_x(0)} u(y-t)\dif y - \int_{B_x(0)} u(y)\dif y  \right)= \frac{1}{\abs{B_x(0)}} \left(  \int_{B_x(-t)} u(y)\dif y - \int_{B_x(0)} u(y)\dif y  \right)= \\
\frac{1}{\abs{B_x(0)}} \left(  \int_{B_x(-t)\setminus B_x(0)} u(y)\dif y - \int_{B_x(0)\setminus B_x(-t)} u(y)\dif y  \right).
\end{aligned}
\end{equation}
Therefore, if we indicate symmetric difference by $\bigtriangleup$, i.e. $A\bigtriangleup B:= (A\cup B)\setminus (A\cap B) $, we can write 
$$
\abs{\widehat{\mathcal{T}_t u}(x)-\widehat{u}(x)}\le 2\norm{u}_\infty \frac{\abs{B_x(-t)\bigtriangleup B_x(0)}}{\abs{B_x(0)}},
$$
and since $\frac{\abs{B_x(-t)\bigtriangleup B_x(0)}}{\abs{B_x(0)}}$ tends to $0$ as $x\to \infty$ (see the next lemma), we conclude that 
$$
\lim_{x\to\infty} \widehat{\mathcal{T}_t u}(x)-\widehat{u}(x)=0
$$
as claimed.

\begin{lem}
Fix $t\in\R^n$; the following quantity (defined for all positive $x\in\R$) 
$$
\frac{\abs{B_x(t)\cap B_x(0)}}{\abs{B_x(0)}}
$$
tends to $1$ as $x\to \infty$.
\end{lem}
\begin{proof}
We will prove this lemma by induction.
If $n=1$ the claim follows easily as, for $x>\abs{t}/2$,
$$
\frac{\abs{B_x(t)\cap B_x(0)}}{\abs{B_x(0)}}= \frac{\abs{\intoo{t-x,t+x}\cap\del{-x,x}}}{\abs{\intoo{-x,x}}}=\frac{2x-\abs{t}}{2x}\longrightarrow 1.
$$
Suppose now that the claim holds true for $n=1,\dots,k$; we are going to prove the claim for $n=k+1$.
Without loss of generality we may assume that $t=(s,0)\in\R^k\times \R$ for some $s\in\R^k$. 
In order to ease the notation a bit, we will consider $s$ fixed and introduce the following functions:
$$
f(x):= \abs{B_x(s)\cap B_x(0)}, \quad g(x):=\abs{B_x(0)} \quad \forall x>0,
$$ 
In other words, the inductive hypothesis for $n=k$ can now be rewritten as 
\begin{equation} \label{hh}
\lim_{x\to\infty} \frac{f(x)}{g(x)}=1.
\end{equation}
Moreover, integrating over the layers (see figure \ref{sfere}) yields the following representation of the claim in the case $n=k+1$:
\begin{equation}\label{k+1}
\lim_{x\to\infty} \frac{\int_{-x}^x f(\sqrt{x^2-r^2})\dif r}{\int_{-x}^x g(\sqrt{x^2-r^2})\dif r}=\lim_{x\to\infty} \frac{\cancel{2}\int_{0}^x f(\sqrt{x^2-r^2})\dif r}{\cancel{2}\int_{0}^x g(\sqrt{x^2-r^2})\dif r}=1
\end{equation}

\begin{figure}[h]
    \centering
    \includegraphics[width=0.6\textwidth]{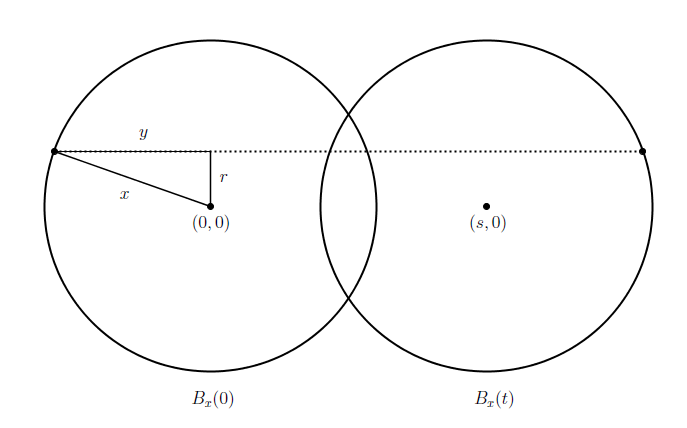}
    \caption{The above construction in the case $n=2$}
    \label{sfere}
\end{figure}

Through the change of variable $y:=\sqrt{x^2-r^2}$, and the introduction of the following auxiliary functions
$$
\varphi(x,y):=\frac{f(y)}{\sqrt{x^2-y^2}}, \quad \psi(x,y):=\frac{g(y)}{\sqrt{x^2-y^2}}\quad \text{(for $0<y<x$),}
$$
the limit in \eqref{k+1} can be rewritten as 
\begin{equation}\label{indclaim}
\lim_{x\to\infty} \frac{\int_0^x \varphi(x,y)\dif y}{\int_0^x \psi(x,y)\dif y}=1.
\end{equation}
By \eqref{hh}, if we fix $\eps>0$ we find a $z>0$ such that, for all $y>z$,   
$$
g(y)(1-\eps)\le f(y)\le g(y)(1+\eps).
$$
Therefore, we can split the integrals in \eqref{indclaim} to get the estimates we need, namely
\begin{equation}\nonumber
\begin{aligned}
\int_0^x \varphi(x,y) \dif y=\int_0^z \varphi(x,y)\dif y+\int_z^x \varphi(x,y)\dif y\le  z\cdot \varphi(x,z)+(1+\eps)\int_z^x \psi(x,y)\dif y,
\end{aligned}
\end{equation}
and 
\begin{equation}\nonumber
\int_0^x \varphi(x,y)\dif y\ge(1-\eps)\int_z^x \psi(x,y)\dif y.
\end{equation}
Similar estimates can be obtained for the the integral containing $\psi$ as well:
$$
\int_z^x \psi(x,y) \dif y \le \int_0^x \psi (x,y) \dif y \le z\cdot \psi(x,z)+\int_z^x \psi(x,y) \dif y.
$$
Therefore,
$$
\frac{(1-\eps)\int_z^x \psi(x,y) \dif y}{z\cdot \psi(x,z)+\int_z^x \psi(x,y) \dif y}       \le 
\frac{\int_0^x \varphi(x,y)\dif y}{\int_0^x \psi(x,y) \dif y} 
\le \frac{z\cdot \varphi(x,z)+(1+\eps) \int_z^x \psi(x,y) \dif y}{\int_z^x \psi(x,y) \dif y}.
$$
Letting $x\to\infty$ we get (notice that $\Lim{x\to\infty} z\cdot \varphi(x,z)= \Lim{x\to\infty} z\cdot \psi(x,z)=0$)
$$
1-\eps \le \liminf_{x\to\infty} \frac{\int_0^x \varphi(x,y)\dif y}{\int_0^x \psi(x,y) \dif y}  \le \limsup_{x\to\infty} \frac{\int_0^x \varphi(x,y)\dif y}{\int_0^x \psi(x,y) \dif y} \le 1+\eps.
$$
We conclude by the arbitrariness of $\eps$.
\end{proof}
Proposition \ref{generalex}, the subsequent remarks and Proposition \ref{serie} clearly hold true in the $n$-dimensional case as well.

\bigskip

\noindent
\textsc{
Research Center for Pure and Applied Mathematics, Graduate
School of
Information Sciences, Tohoku University, Sendai 980-8579
, Japan.
} \\
\noindent
{\em Electronic mail address:}
cava@ims.is.tohoku.ac.jp

\end{document}